\documentclass[reqno]{amsart}
\usepackage[T1]{fontenc}
\usepackage{dsfont}
\usepackage{mathrsfs}
\usepackage[colorlinks]{hyperref}
\usepackage{xcolor}
\usepackage[a4paper,asymmetric]{geometry}
\usepackage{mathscinet}
\usepackage{latexsym}
\usepackage{amsthm}
\usepackage{amssymb}
\usepackage{amsfonts}
\usepackage{amsmath}
\usepackage{longtable}
\usepackage{graphicx}
\usepackage{multirow}
\usepackage{multicol}

\newtheorem{theorem}{Theorem}[section]
\newtheorem{thm}[theorem]{Theorem}

\newtheorem{lem}[theorem]{Lemma}
\newtheorem{remark}[theorem]{Remark}
\newtheorem{proposition}[theorem]{Proposition}
\newtheorem{prop}[theorem]{Proposition}
\newtheorem{corollary}[theorem]{Corollary}

\theoremstyle{definition}
\newtheorem{definition}[theorem]{Definition}
\newtheorem{defn}[theorem]{Definition}

\theoremstyle{remark}
\numberwithin{equation}{section}

 \DeclareMathAlphabet{\mathpzc}{OT1}{pzc}{m}{it}
 \DeclareMathAlphabet{\mathsfsl}{OT1}{cmss}{m}{sl}

  \newcommand{\DR}{\mathbb{D}}
  \newcommand{\FH}{\mathfrak{H}}

\newcommand{\abs}[1]{\left\vert#1\right\vert}
\newcommand{\set}[1]{\left\{#1\right\}}

\newcommand{\norm}[1]{\left\Vert#1\right\Vert}

\newcommand{\E}{\mathbb{E}}

 \newcommand{\tensor}[1]{\mathsf{#1}}

 \newcommand{\Rnum}{\mathbb{R}}
 \newcommand{\Cnum}{\mathbb{C}}
 
 \newcommand{\Nnum}{\mathbb{N}}

 \newcommand{\innp}[1]{\langle {#1}\rangle}
\newcommand{\Be}{\begin{equation}}
\newcommand{\Ee}{\end{equation}}
\newcommand{\Bs}{\begin{split}}
\newcommand{\Es}{\end{split}}
\newcommand{\Bes}{\begin{equation*}}
\newcommand{\Ees}{\end{equation*}}
\newcommand{\BT}{\begin{thm}}
\newcommand{\ET}{\end{thm}}
\newcommand{\Bp}{\begin{proof}}
\newcommand{\Ep}{\end{proof}}
\newcommand{\BL}{\begin{lem}}
\newcommand{\EL}{\end{lem}}
\newcommand{\BP}{\begin{proposition}}
\newcommand{\EP}{\end{proposition}}
\newcommand{\BC}{\begin{corollary}}
\newcommand{\EC}{\end{corollary}}
\newcommand{\BR}{\begin{remark}}
\newcommand{\ER}{\end{remark}}
\newcommand{\BD}{\begin{defn}}
\newcommand{\ED}{\end{defn}}
\newcommand{\BI}{\begin{itemize}}
\newcommand{\EI}{\end{itemize}}

\allowdisplaybreaks

\begin{document}
\title[Complex Malliavin Calculus]{A note on the Moment of Complex Wiener-It\^o Integrals}
\author[Y. Chen]{Yong CHEN}
\address{School of Mathematics and Statistics, Hunan university of science and technology, Xiangtan, 411201, Hunan, China}
\email{chenyong77@gmail.com; zhishi@pku.org.cn}
\author[G. Jiang]{Guo JIANG}
\address{Corresponding author: School of Mathematics and Statistics, Hubei Normal University, Huangshi, Hubei 435002, China}
\email{jg1996@126.com }
\begin{abstract}
For a sequence of complex  Wiener-It\^o multiple integrals, the equivalence between the convergence of the symmetrized contraction norms and that of the non-symmetrized contraction norms is shown directly by means of a new version of complex Mallivian calculus using the Wirtinger derivatives of complex-valued functions.
\end{abstract}
\maketitle

\section{ Introduction}\label{sec 03}

Recently, fourth moment theorems are extended to the case of complex multiple stochastic integrals with different methods \cite{camp,chw 17, cl2}. S.Campese \cite{camp} uses stein's method for a general context of Markov diffusion generators. \cite{cl2} is essentially by reduction to the two-dimensional real-valued case. \cite{chw 17} is an adaption of the classical arguments by D. Nualart, G.Peccati and S. Ortiz-Latorre for the one-dimensional real-valued case in \cite{np,nourorti}. That is to say, in \cite{chw 17} they show the five equivalent conditions by means of $(\mathrm{i})\Rightarrow(\mathrm{ii})\Rightarrow (\mathrm{iii})\Rightarrow (\mathrm{iv})\Rightarrow (\mathrm{v})\Rightarrow (\mathrm{i})$.

Since in the real case there is a direct and short proof \cite[p100]{np} for the equivalence between conditions $(\mathrm{iii})$ and $(\mathrm{iv})$, i.e., the convergence of the symmetrized contraction norms and that of the non-symmetrized contraction norms, the question naturally arises whether there is still a direct proof to that equivalence in the complex case. The key aim of this note is to give an affirmative answer to the above question.  

To state the theorem we denote $\FH$ a complex separable Hilbert space with inner product $\innp{\cdot,\,\cdot}_{\FH}$ and norm $\norm{\cdot}_{\FH}$ and let $Z=\set{Z(h):\,h\in  \FH}$ be a complex isonormal Gaussian process over $\FH$, i.e., the complexification of the classical real isonormal Gaussian process (see Example 1.9 of \cite{janson} or Definition 2.6 of \cite{cl2}). The complex Wiener-Ito (multiple) integrals is an isometric mapping $I_{m,n}$ from $\FH^{\odot m}\times \FH^{\odot n}$ to $L^2(\Omega,\sigma(Z))$ (see Definition 2.10 of \cite{cl2}).
Now the theorem is stated as follows.
\begin{thm}\label{main law}
Let $\set{F_{k}=I_{m,n}(f_k)}$ with $f_k\in \FH^{\odot m}\otimes \FH^{\odot n}$  be  a  sequence of  $(m,n)$-th complex Wiener-It\^{o} multiple integrals, with $m$ and $n$ fixed and
$m+n\ge 2$. Then the following statements are equivalent:
\begin{itemize}
\item[\textup{(iii)}] $\norm{f_k\otimes_{i,j} f_k}_{\FH^{\otimes ( 2(l-i-j))}}\to 0$ and $\norm{f_k{\otimes}_{i,j} h_k}_{\FH^{\otimes ( 2(l-i-j))}}\to 0 $ for any $0<i+j\le l-1$ where $l=m+n$ and $h_k$ is the kernel of $\bar{F}_k$, i.e., $\bar{F}_k=I_{n,m}(h_k)$.
\item[\textup{(iv)}] $\norm{f_k\tilde{\otimes}_{i,j} f_k}_{\FH^{\otimes ( 2(l-i-j))}}\to 0$ and $\norm{f_k\tilde{{\otimes}}_{i,j} h_k}_{\FH^{\otimes ( 2(l-i-j))}}\to 0 $ for any $0<i+j\le l-1$.
\end{itemize}
\end{thm}
The proof of the above theorem is a direct application of the following proposition which gives an expression of the fourth moment of a complex Wiener-Ito integral by means of the sum of the inner products of some symmetrized contractions.
\begin{prop}\label{prop 3-2}
Suppose that $F=I_{m,n}(f)$ with $f\in \FH^{\odot m}\otimes \FH^{\odot n}$ and that $\bar{F}=I_{n,m}(h)$. Then
\begin{align}
   & \E[\abs{F}^4]-2\big(\E[\abs{F}^2]\big)^2 -\abs{E[{F}^2]}^2\nonumber\\
      &=2\sum_{r=1}^{l-1}\big[(l-r)!\big]^2
   \innp{\vartheta_r,\psi_r}_{\FH^{\otimes 2(l-r)}}+ \sum_{r=1}^{l'-1}(2m-r)!(2n-r)!\innp{\varsigma_r,\, \varphi_r}_{\FH^{\otimes 2(l-r)}},\label{df ff1}
\end{align}
where  $l=m+n,\,l'=2(m\wedge n)$ and 
\begin{align}
   \vartheta_r
     &= \sum_{i+j=r}\frac{i}{m} {m \choose i }^2{n\choose j}^2 i!j!\,f\tilde{\otimes}_{i,j}h,\\
\psi_r&=\sum_{i+j=r} {m\choose i}^2{n\choose j}^2 i!j!\, f\tilde{\otimes}_{i,j} h,\label{bbb} \\
    \varsigma_r&= \sum_{i+j=r}\frac{i}{m} {m \choose i }{n\choose i}{m \choose j }{n\choose j} i!j!\,f\tilde{\otimes}_{i,j}f,\\
   \varphi_r&=\sum_{i+j=r} {m\choose i}{n\choose i}{n\choose j}{m\choose j} i!j!\, f\tilde{\otimes}_{i,j} f.
\end{align}
\end{prop}
Similar to the real case \cite[p97]{np}, the key idea of the proof of Proposition~\ref{prop 3-2} is using the complex Mallivian calculus. We have to exploit a new version of complex Malliavin derivative $D$, its adjoint operator $\delta$ and a complex Ornstein-Uhlenbeck operator $L=\delta D$ which is distinct from the known versions of complex Mallivian calculus in \cite{bauNua} or \cite{janson}.

\section{Preliminaries: Concise Complex Malliavin Calculus}
\subsection{Malliavin derivative operators}
The following definition of complex Malliavin derivatives which makes use of the Wirtinger derivatives of complex-valued functions is distinct from what the authors defined in \cite{bauNua} or \cite{janson} and is easier to use in our case.
\begin{defn}\label{dfn32}
Let $\mathcal{S}$ denote the set of all random variables of the form
\begin{equation}\label{definition}
  f\big(Z(\varphi_1),\cdots,Z(\varphi_m)\big),
\end{equation}
where $f\in C_{\uparrow}^\infty(\Cnum^m)$ and $\varphi_i\in\FH,i=1,2,\cdots,m$. 
If $F\in \mathcal{S}$, then the complex Malliavin derivatives of $F$ (with respect to $\zeta$) are the elements of $L^2(\Omega, \mathfrak{H} )$ defined by \cite{chw 17,cl2}:
 \begin{align}
    D  F     & =\sum_{i=1}^m \partial_{i } f(Z(\varphi_1),\dots,Z(\varphi_m))\varphi_{i},\\
    \bar{D}  F     & =\sum_{i=1}^m \bar{\partial}_{i } f(Z(\varphi_1),\dots,Z(\varphi_m))\bar{\varphi}_{i},
  \end{align} where
\begin{equation*}
   \partial_j f= \frac{\partial}{\partial z_j}f(z_1,\dots,z_m) ,\quad \bar{\partial}_j f=\frac{\partial}{\partial \bar{z}_j}f(z_1,\dots,z_m),\quad j=1,\dots,m
\end{equation*} are the Wirtinger derivatives.
\end{defn}
The above definition implies that $ \overline{D F}=\bar{D}\bar{F}$. The following proposition gives an integration by parts formula of complex Gaussian random variables, whose proof is straight forward. Please refer to Lemma 3.2 of \cite{cl} or Lemma 2.3 of \cite{camp}.
\begin{prop}[integration by parts formula]\label{prop 2_3}
   Suppose that $F\in\mathcal{S}$ and $h\in \FH$, then we have the following integration by parts formula
   \begin{align*}
        E[Z(h)\times \bar{F}]=E[\innp{h,\, D F}],\qquad  E[ \bar{Z} (h)\times \bar{F}]=E[\innp{h,\,\bar{D} F}]) .
   \end{align*}
\end{prop}

It is routine to show that $ D$ and $\bar{D}$ are closable from $L^p(\Omega)$ to $L^p(\Omega,\FH)$. Denote by $\DR^{1,p}$ and $\bar{\DR}^{1,p}$ the closure of $\mathcal{S}$ with respect to the Soblev seminorm $\norm{\cdot}_{1,p}$.The following proposition is an adaption of the real-valued case which gives a sufficient condition to check a random belonging to the domain $\DR^{1,2}$ or $\bar{\DR}^{1,2}$, please see for example \cite{Nualart}.
\begin{prop}\label{jprop}
   Let $\set{F_n,n\ge 1}$ be a sequence of random variable in $\DR^{1,2}$ (resp. $\bar{\DR}^{1,2}$) that converges to $F$ in $L^2(\Omega)$ and that
$$\sup\limits_nE[\norm{DF_n}_{\mathfrak{H}}^2]<\infty, \quad (\text{resp.  } \sup\limits_nE[\norm{\bar{D}F_n}_{\mathfrak{H}}^2]<\infty),$$
then $F $ belongs to $\DR^{1,2}$ (resp. $\bar{\DR}^{1,2}$) and the sequence of derivatives $DF_n$ (resp. $\bar{D}F_n$) converges weakly to $DF$ (resp. $\bar{D}F$) in $L^2(\Omega,\mathfrak{H})$.
\end{prop}

By the chain rules of Wirtinger derivatives \cite{camp}, we obtain the following chain rules of complex Malliavin derivatives.
\begin{prop}{\bf (Chain rule)}\label{prop312}
 If $\varphi:\Cnum^m\to \Cnum$ is a continuously differentiable function with bounded partial derivatives and if $F=(F^1,\dots,F^m)$ is a random vector whose components are elements of $\DR^{1,2}\bigcap \bar{\DR}^{1,2}$, then $\varphi(F)\in \DR^{1,2}\bigcap \bar{\DR}^{1,2}$ and
\begin{align}
   D\varphi(F)&=\sum_{j=1}^m \partial_j \varphi(F)DF^j+\bar{\partial}_j \varphi(F)D\overline{ F^j},\label{chr1}\\
   \bar{D}\varphi(F)&=\sum_{j=1}^m \partial_j \varphi(F)\bar{D}F^j+\bar{\partial}_j \varphi(F)\bar{D}\overline{ F^j}.\label{chr2}
\end{align}
\end{prop}
\begin{remark}
To compare with Theorem 15.34 of \cite[p238]{janson}, we find that our definitions of complex Malliavin derivative are different with Janson's Definition 15.26 \cite[p236]{janson}.
\end{remark}
 We define the divergence operators $\delta$ and $\bar{\delta} $ as the adjoint of $D$ and $\bar{D}$ respectively, with the domains $\mathrm{Dom}( \delta) $ and $\mathrm{Dom}( \bar{\delta}) $ the subsets of $L^2(\Omega,\FH )$ composed of those elements $u$ such that there exists a constant $c>0$ verifying for all $ F\in \mathcal{S}$,
   \begin{align*}
     \big|E[\innp{D F,u}]\big|\le c\norm{F}, \quad(\text{resp. } \big|E[\innp{\bar{D} F,u}]\big|\le c\norm{F}).
   \end{align*}
   If $u\in\mathrm{Dom}( \delta ) $ or $u\in\mathrm{Dom}( \bar{\delta} )$, then $ \delta  u$ and $\bar{\delta}u$ are the unique element of $L^2(\Omega)$ given respectively by the following duality formula: for all $ F\in \mathcal{S}$,
   \begin{align}\label{dul form}
      E[(\delta  u)\times \bar{F}]=E[\innp{u,\, D F}],\quad (\text{resp. }
      E[( \bar{\delta}  u)\times \bar{F}]=E[\innp{u,\,\bar{D} F}]) .
   \end{align}

\subsection{Complex Ornstein-Uhlenbeck operators}
We define complex Ornstein-Uhlenbeck operators which are different with that in \cite{janson}.
\begin{definition}
   Complex Ornstein-Uhlenbeck operators are defined as
\begin{align}
   \tensor{L}=\delta D,\qquad \bar{\tensor{L}}=\bar{\delta} \bar{D}.
\end{align}
\end{definition}

\begin{prop}\label{prop d imn}
Suppose that ${I}_{m,n}(f)$ is the complex Wiener-It\^{o} integral of $f$ with respect to $Z$ for any $ f \in \FH^{\odot m}\otimes \FH^{\odot n}$. Then we have that
\begin{align}
   D_{\cdot}(I_{m,n}(f))&=m I_{m-1,n}(f),\,\quad
   \bar{D}_{\cdot}(I_{m,n}(f))=nI_{m,n-1}(f),\label{d i mn}\\
   \tensor{L}(I_{m,n}(f) )&=m I_{m,n}(f),\qquad \bar{\tensor{L}}(I_{m,n}(f) )=n I_{m,n}(f).\label{l i mn}
\end{align} %
\end{prop}
\begin{proof}First, we claim that a complex Hermite polynomials $ J_{m,n}(z,\rho)$ \cite{ito,cl} satisfies that
 \begin{itemize}
    \item[\textup{1)}]       partial derivatives:  \begin{align}
     \frac{\partial}{\partial z} J_{m,n}(z,\rho)&=m J_{m-1,n}(z,\rho),\qquad
      \frac{\partial}{\partial \bar{z}} J_{m,n}(z,\rho)=n J_{m,n-1}(z,\rho),\label{partizz}\\
      \frac{\partial}{\partial \rho} J_{m,n}(z,\rho)&=-mn J_{m-1,n-1}(z,\rho).\label{parti rho}
  \end{align}
    \item[\textup{2)}] recursion formula:
    \begin{align}
       J_{ m+1,n}(z,\rho)= {z}J_{m,n}(z,\rho)-n\rho J_{m,n-1}(z,\rho),\label{jm1n}\\
       J_{m,n+1}(z,\rho)=\bar{z}J_{m,n}(z,\rho)-m\rho J_{m-1,n}(z,\rho). \label{jmn1}
    \end{align}
    \end{itemize}
    In fact, about Eq.(\ref{partizz}), please refer to Theorem 12 (D) of \cite{ito} or Proposition A.6 of \cite{cl}. Eq.(\ref{parti rho}) is obtained by taking partial derivative $\frac{\partial}{\partial \rho} $ in both sides of the generating function of complex Hermite polynomials. Eq.(\ref{jm1n})-(\ref{jmn1}) are shown in Theorem 12 (C) of \cite{ito} and \cite[p15]{cl}.

Second, suppose $f=h^{\otimes m}\otimes \bar{h}^{\otimes n}$with $h\in \FH$. Denote $\rho=\norm{h}^2$ and $\vec{t}^k=(t_1,\dots,t_k)$, $\vec{s}^k=(s_1,\dots,s_k)$. Then we obtain that
\begin{align*}
   D_{\cdot}(I_{m,n}(f))&=D_{\cdot}(J_{m,n}(Z(h),\rho))\\
   &=m J_{m-1,n}(Z(h),\rho)h(\cdot)\\
   &=m I_{m-1,n}(h^{\otimes m-1}\otimes \bar{h}^{\otimes n})h(\cdot)\\
   &=m I_{m-1,n}f(\vec{t}^{m-1},\cdot,\vec{s}^n).
\end{align*} Denote $G=I_{m-1,n}( h^{\otimes m-1}\otimes \bar{h}^{\otimes n})$, then we have that ${D}\bar {G}=nI_{n-1,m-1} (h^{\otimes n-1}\otimes \bar{h}^{\otimes m-1})h $ and that
\begin{align*}
   \tensor{L}(I_{m,n}(f) )&=m \delta (G h )\\
   &=m [GZ(h)-\innp{h,\,{D}\bar {G}}_{\FH}]\\
   &=m[Z(h) J_{m-1,n}(Z(h),\rho)-n\rho J_{m-1,n-1}(Z(h),\rho)]\\
   &=m J_{m,n}(Z(h),\rho)\\
   &=mI_{m,n}(f).
\end{align*}
 Similarly, we have that $\bar{D}_{\cdot}(I_{m,n}(f))=nI_{m,n-1}f(\vec{t}^{m},\vec{s}^{n-1},\cdot)$ and that $\bar{\tensor{L}}(I_{m,n}(f) )=n I_{m,n}(f)$.

Finally, by means of density arguments (or the polarization technique), it is easily to show that
(\ref{d i mn})-(\ref{l i mn}) hold.
\end{proof}

\section{Proof of the main thoerems}
To compare Lemma~2.3 of \cite{chw 17} with our findings, we list it as follows.
\begin{lem}\label{lem 3-1}
Suppose that $F=I_{m,n}(f)$ with $f\in \FH^{\odot m}\otimes \FH^{\odot n}$ and that $\bar{F}=I_{n,m}(h)$. Then
\begin{align*}
   & \E[\abs{F}^4]-2\big(\E[\abs{F}^2]\big)^2 -\abs{E[{F}^2]}^2 \\
      &=\sum_{0 <i+j<l}{m\choose i}^2{n\choose j}^2 (m!n!)^2  \norm{f\otimes_{i,j}f}^2_{\FH^{\otimes(2(l-i-j))}} +\sum_{r=1}^{l-1}((l-r)!)^2\norm{\psi_r}^2_{\FH^{\otimes(2(l-r))}} \\
      &=\sum_{0<i+j<l'}{m\choose i}{n\choose i}{n\choose j}{m\choose j} (m!n!)^2 \norm{f\otimes_{i,j}h}^2_{\FH^{\otimes(2(l-r))}}  +\sum_{r= 1}^{l-1} (2m-r)!(2n-r)!  \norm{\varphi_r}^2_{\FH^{\otimes 2(l-r)}},
\end{align*}
where $l=m+n,\,l'=2(m\wedge n)$ and $\psi_r,\,\varphi_r$ are as in Proposition~\ref{prop 3-2}.
\end{lem}
\noindent{\it Proof of Proposition~\ref{prop 3-2}.\,}
We divide the proof into several steps.

  Step 1: We claim that
  \begin{align}
   \frac{1}{m}\E\big[\abs{F}^2\norm{D F}^2_{\FH}\big]=\big(\E[\abs{F}^2]\big)^2+\sum_{r=1}^{m+n-1}\big[(m+n-r)!\big]^2
   \innp{\vartheta_r,\psi_r}_{\FH^{\otimes 2(m+n-r)}} .\label{first part}
  \end{align}
  In fact, it follows from the product formula of complex Wiener-It\^o integrals \cite{ch 17} and the Fubini theorem that
  \begin{align}
    \frac{1}{m}\norm{D F}^2_{\FH}&=m \norm{I_{m-1,n}(f)}^2_{\FH}\nonumber \\
    &=m \sum_{i=0}^{m-1}\sum^{n}_{j=0}\,{m-1\choose i}^2{n\choose j}^2 i!j!\,I_{m+n-1-i-j,m+n-1-i-j}(f\otimes_{i+1,j}h)\nonumber \\
    &=m \sum_{i=1}^{m}\sum^{n}_{j=0}\,{m-1\choose i-1}^2{n\choose j}^2 (i-1)!j!\,I_{m+n-i-j,m+n-i-j}(f\otimes_{i,j}h)\nonumber \\
    &=\E[\abs{F}^2]+\sum_{r=1}^{m+n-1}I_{m+n-r,m+n-r}(\vartheta_r).\label{df square1}
  \end{align}
  On the other hand, we can obtain that
  \begin{align}\label{f squre}
     \abs{F}^2&=\sum_{i=0}^{m}\sum^{n}_{j=0}\,{m\choose i}^2{n\choose j}^2 i!j!\,I_{m+n-i-j,m+n-i-j}(f\otimes_{i,j}h)\nonumber\\
     &=\E[\abs{F}^2]+\sum_{r=0}^{m+n-1}I_{m+n-r,m+n-r}(\psi_r).
  \end{align}
  Substituting (\ref{df square1}) and (\ref{f squre}) into the left side of (\ref{first part}) and using the orthogonality properties of multiple integrals, we have that (\ref{first part}) holds.

  Step 2: We claim that
  \begin{align}\label{second part}
    \frac{1}{m}\E\big[\innp{D F,\,D\bar{F}}_{\FH}\bar{F}^2\big]= \abs{\E[{F}^2]}^2+\sum_{r=1}^{2(m\wedge n)-1}(2m-r)!(2n-r)!\innp{\varsigma_r,\, \varphi_r}_{\FH^{\otimes 2(m+n-r)}}.
  \end{align}
  In fact, the product formula and the Fubini theorem implies that
  \begin{align}
   \frac{1}{m}\E\big[\innp{D F,\,D\bar{F}}_{\FH}&=n\innp{I_{m-1,n}(f),\,I_{n-1,m}(h)}_{\FH}  \nonumber\\
   &=n \sum_{i=0}^{m\wedge n-1}\sum^{m\wedge n}_{j=0}\,{m-1\choose i}{n-1\choose i}{m\choose j}{n\choose j}i!j!\,I_{2m-1-i-j,2n-1-i-j}(f\otimes_{i+1,j}f)\nonumber\\
   &=n \sum_{i=1}^{m\wedge n }\sum^{m\wedge n}_{j=0}\,{m-1\choose i-1}{n-1\choose i-1}{m\choose j}{n\choose j}(i-1)!j!\,I_{2m-i-j,2n -i-j}(f\otimes_{i,j}f)\nonumber\\
   &=\E[{F}^2] +\sum_{r=1}^{2(m\wedge n)-1}I_{2m-r,2n-r}(\varsigma_r).\label{df square12}
  \end{align}
  On the other hand, we can obtain that
  \begin{align}
     F^2&=\sum_{i,j=0}^{m\wedge n}\,{m\choose i}{n\choose i}{m\choose j}{n\choose j}i!j!\,I_{2m-i-j,2n-i-j}(f\otimes_{i,j}f)\nonumber\\
     &=\E[{F}^2]+\sum_{r=1}^{2(m\wedge n)-1}I_{2m-r,2n-r}(\varphi_r).\label{f squre 2}
  \end{align}
  Substituting (\ref{df square12}) and (\ref{f squre 2}) into the left side of (\ref{second part}) and using the orthogonality properties of multiple integrals, we have that (\ref{second part}) holds.

Step 3: By approximation, we claim that for any Wiener-Ito integral $F=I_{m,n}(f)$,
\begin{align}\label{qiu dao}
  D(\bar{F}F^2)=2\abs{F}^2 DF+F^2 D\bar{F}.
\end{align}
In fact, for the function $g(z)=\bar{z}z^2$ and $n\in \Nnum$, we take $$g_n=g\cdot \big[\chi_{[-n,n]}+k(-n-x)+k(n+x)\big]$$
 where $\chi_A(\cdot)$ the index function of a set $A$ and $k(x)=e^{-\frac{1}{x(1-x)}}\chi_{(0,1)}(x)$ a cut-off function. For any $p\ge 1$, $g_n\in C^{\infty}_c(\Rnum^2)$ and $g_n,\partial g_n,\,{\partial}\bar{g}_n$ converge to $g,\,\partial g,\,{\partial}\bar{g}$ respectively in the sense of $L^p(\mu)$ with $F\sim \mu$. The chain rule, i.e, Proposition~\ref{prop312}, implies that
 \begin{align*}
   D(g_n(F))=\partial g_n(F) DF+\bar{\partial} g_n(F) D\bar{F}.
 \end{align*} The hypercontrativity inequality of Wiener-Ito integrals (see Proposition~2.4 of \cite{ch 17}) and the Cauchy-Schwarz inequality imply that as $n\to \infty$, in the sense of $L^2(\Omega,\,\FH)$,
 \begin{align*}
   \partial g_n(F) DF+\bar{\partial} g_n(F) D\bar{F} \to 2\abs{F}^2 DF+F^2 D\bar{F}.
 \end{align*}
Then we obtain (\ref{qiu dao}) by Proposition~\ref{jprop}.

Step 4: It follows from Proposition~\ref{prop d imn}, the dual relation and the chain rule that
\begin{align*}
  \E[\abs{F}^4]&=\E[ {F}\overline{\bar{F}F^2}]\\
  &=\frac{1}{m}\E[ \delta D {F}\times \overline{\bar{F}F^2}]\\
  &=\frac{1}{m}\E[ \innp{ D {F},\, D(\bar{F}F^2)}_{\FH}]\\
  &=\frac{1}{m}\big[ 2 \norm{DF}_{\FH}^2\times\abs{F}^2 +\innp{DF,\,D\bar{F}}_{\FH} \times\bar{F}^2\big].
\end{align*}
By substituting (\ref{first part}) and (\ref{second part}) into the above equality displayed, we obtain (\ref{df ff1}).
{\hfill\large{$\Box$}}\\


\noindent{\it Proof of Theorem~\ref{main law}.\,}
(iii) implies (iv) is elementary. Now suppose that (iv) holds. Then the Cauchy-Schwarz inequality implies that as $k\to \infty$,
\begin{align*}
  &\abs{ \innp{\vartheta_r,\psi_r}_{\FH^{\otimes 2(l-r)}}}\\
  &\le \sum_{i+j=r}\sum_{i'+j'=r}\frac{i}{m} {m \choose i }^2{n\choose j}^2 i!j!\, {m\choose i'}^2{n\choose j'}^2 i'!j'!\,\abs {\innp{ f_k\tilde{\otimes}_{i,j}h_k ,\,f_k\tilde{\otimes}_{i',j'} h_k}_{\FH^{\otimes 2(l-r)}}}\\
   &\le \sum_{i+j=r}\sum_{i'+j'=r}\frac{i}{m} {m \choose i }^2{n\choose j}^2 i!j!\, {m\choose i'}^2{n\choose j'}^2 i'!j'!\,\norm{f_k\tilde{\otimes}_{i,j}h_k}_{\FH^{\otimes 2(l-r)}} \norm{f_k\tilde{\otimes}_{i',j'} h_k}_{\FH^{\otimes 2(l-r)}}\\
   &\to 0.
\end{align*}
In the same way, we can obtain that as $k\to \infty$,
\begin{align*}
 \abs{ \innp{\varsigma_r,\, \varphi_r}_{\FH^{\otimes 2(l-r)}}}\to 0.
\end{align*}
 Proposition~\ref{prop 3-2} combining with the above two equalities displayed implies that as $k\to \infty$,
 \begin{align*}
   \E[\abs{F}^4]-2\big(\E[\abs{F}^2]\big)^2 -\abs{E[{F}^2]}^2 \to 0,
 \end{align*}
 which implies that (iii) holds from Lemma~\ref{lem 3-1}.
{\hfill\large{$\Box$}}\\

\vskip 0.2cm {\small {\bf  Acknowledgements}:
Y. Chen is supported by the China Scholarship Council (201608430079);
G. Jiang is supported by the Hubei Provincial NSFC (2016CFB526).


\end{document}